\documentclass{article}

\usepackage{arxiv}

\usepackage[utf8]{inputenc} 
\usepackage[T1]{fontenc}    
\usepackage{fancyhdr}
\usepackage{lmodern}
\usepackage[colorlinks]{hyperref}  
\usepackage{url}            
\usepackage{booktabs}       
\usepackage{amsfonts}       
\usepackage{nicefrac}       
\usepackage{microtype}      
\usepackage{lipsum}

\usepackage{graphicx}
\usepackage{bm}
\usepackage{epsfig}
\usepackage{epstopdf}
\usepackage{subfigure}

\newcommand{\vir}[1]{``#1''}

\usepackage{enumerate}
\usepackage{amssymb}
\usepackage{amsmath}
\usepackage{amsthm}
\usepackage{enumitem}
\usepackage{mathrsfs}

\renewcommand{\bar}{\overline}
\renewcommand{\rho}{\varrho}
\renewcommand{\phi}{\varphi}

\theoremstyle{plain}
\theoremstyle{definition} 
\newtheorem{theorem}{Theorem}[section] 
\newtheorem{cor}{Corollary}[section] 
\newtheorem{lem}{Lemma}[section] 
\newtheorem{prop}{Proposition}[section]

\newcommand{\Ra}{\mathsf{R}}
\DeclareMathOperator{\esssup}{ess\,sup}

\title{Existence and uniqueness of slightly compressible Boussinesq's flow in Darcy-Bénard problem}

\author{  
G. Arnone\href{https://orcid.org/0000-0002-3317-6358}{\includegraphics[scale=0.1]{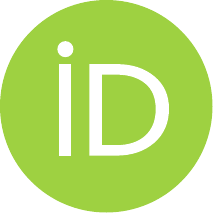}} \\ Dipartimento di Matematica e Applicazioni 'R.Caccioppoli' \\ Università degli Studi di Napoli Federico II \\ Via Cintia, Monte S.Angelo, 80126 Napoli \\ Italy \\   
\texttt{giuseppe.arnone@unina.it} \\ 
\And
F. Capone\thanks{Corresponding author.} \href{https://orcid.org/0000-0002-0672-999X}{\includegraphics[scale=0.1]{orcid}} \\ Dipartimento di Matematica e Applicazioni 'R.Caccioppoli' \\ Universit\'a degli Studi di Napoli Federico II \\ Via Cintia, Monte S.Angelo, 80126 Napoli \\ Italy \\ 
\texttt{fcapone@unina.it} }

\renewcommand{\shorttitle}{Existence and uniqueness of slightly compressible Boussinesq's flow}

\begin{document}
\maketitle

\begin{abstract}
In the present paper, we study the existence, uniqueness and behaviour in time of the solutions to the Darcy-Bénard problem for an extended-quasi-thermal-incompressible fluid-saturated porous medium uniformly heated from below. Unlike the classical problem, where the compressibility factor of the fluid vanishes, in this paper we allow the fluid to be slightly compressible and we address the well-posedness analysis for the full nonlinear initial boundary value problem for the perturbed system of governing equations modelling the convection in porous media phenomenon. 
\end{abstract}

\section{Introduction}

Thermal convection phenomena in Newtonian fluid-saturated porous media heated from below, are in essence physically motivated by buoyant force induced from density variations due to the presence of a thermal gradient along the layer. Indeed, in non-isothermal processes variations in temperature generate variations in the fluid's properties, e.g. density and viscosity. An analysis including the full effect of these variations is so complicated that some approximations become essential. The vast majority studies concerning the stability of basic steady state motions in porous media, see e.g. \cite{stabilityandwavemotion2004,NieldBejan2017} and referenced therein, are addressed under a celebrated hypothesis, known as the \emph{Boussinesq}, or \emph{Oberbeck-Boussinesq} OB \emph{approximation}, \cite{oberbeck1879,boussinesq1903}. In the words of Boussinesq himself, the approximation assume that: \vir{\emph{The variations of density can be ignored except where they are multiplied by the acceleration of gravity in equation of motion for the vertical component of the velocity vector}}, \cite{boussinesq1903}. The crucial consequence of this assumption, called in 1916 by Lord Rayleigh \emph{Boussinesq approximation}, see \cite{Ray1916}, is the possibility to deal with a \emph{quasi}-incompressible system of coupled dynamic and thermal equations, where the buoyancy is the main driving force. It must be noted that the underlying nature of the OB approximation had attracted considerable attention among scholars. In particular, the topic has been discussed in depth in the thermodynamic and continuum mechanics frameworks and those insights can be found in some brilliant papers like e.g. \cite{Muller1985,GouMuracRogers2011,gouin2012consistent} and \cite{rajagopal1996oberbeck,rajagopal2009oberbeck,rajagopal2015approximation}, respectively. While we shall not go into a detailed discussion of any of the above papers, a discursive summary of the crucial contribution in the thermodynamic framework can be concisely addressed. In his 1985 monograph, M\"uller proved that as long as one assume as \emph{incompressible} a medium whose constitutive equations do not depend on pressure, then the only density function that wouldn't be at odd with the Gibbs law is a constant density. This conclusion is in disagreement with with empirical observations, according to which fluids expands when heated, and the theoretical assumptions such as the OB approximation. For this reason the above fact was called, in \cite{GouMuracRogers2011}, \emph{M\"uller paradox}. In order to fix this contradiction, a less restrictive model of incompressibility was proposed by Gouin, Muracchini and Ruggeri in \cite{GouMuracRogers2011}. In this paper the Authors proposed an alternative model which requires that \emph{the only} constitutive function independent of the pressure is the density and as a result the Gibbs equation is satisfied as long as the pressure involved in the process is below  a critical pressure value. This material was named \emph{quasi-thermal-incompressible medium}. This incompressibility model was obtained as a limit process justifying the compatibility between incompressibility and the Gibbs relation under a specific limitation on the pressure value involved in the process. However, we must stress that quasi-thermal incompressibility does not characterize a real compressible material for which the chemical potential must be a concave function of the pressure and temperature. In this regard, Gouin and Ruggeri in \cite{gouin2012consistent}, enforcing some essential thermodynamic conditions (namely \emph{entropy principle} and \emph{thermodynamic stability}) introduce the class of \emph{extended-quasi-thermal-incompressible fluids}, and they proposed as a significant case to take into account the following density law:
\begin{equation}\label{seitu}
	 \rho(p,T)=\rho_0[1-\alpha(T-T_0)+\beta(p-p_0)], 
\end{equation}
where $ p $ and $ T $ are the pressure and temperature fields, respectively, $ \rho_0 $, $ T_0 $ and $ p_0 $ are the reference density, temperature and pressure value, respectively, and $ \alpha $ and $ \beta $ are the thermal expansion coefficient, and compressibility factor, defined respectively by
\begin{equation}\label{key}
	 \alpha=\dfrac{1}{V}V_T,\quad \beta=-\dfrac{1}{V}V_p. 
\end{equation}
One consequence of introducing this more comprehensive scheme is that the well-posedness of the corresponding mathematical model requires the pressure to be treated as an \emph{independent} unknown instead as a Lagrange multiplier associated with the incompressibility constraint. For this reason the pressure will satisfy a suitable elliptic problem and will be subjected to Robin boundary conditions. The well-posedness and stability of convective solutions in the classical Bénard problem for slightly compressible fluids is addressed in \cite{corli2019benard,de2019existence,passerini2020benard}. Moreover, let us note that in recent time, the constitutive density \eqref{seitu} attracted considerable attention in the framework of hydrodynamic stability, in particular because the extra buoyancy term gives a more accurate description of fluids which are slightly compressible. From recent investigations addressed in the contest of thermal convection problems in clear fluids and fluid-saturated porous media, see e.g. \cite{passerini2014benard,corli2019benard}, it is proved that the compressibility factor has a \emph{destabilizing effect} on the onset of convection. \\
The objective of the present paper is to investigate the well-posedness of the initial-boundary value problem modelling the nonlinear bi-dimensional perturbation of the steady state solution of a slightly compressible fluid-saturated porous medium. More specifically, the paper is organized as follows. In section \ref{eu1} the Darcy-Bénard problem for slightly incompressible fluid-saturated porous media is introduced, and the steady state solution as well as the perturbed non-dimensional system are computed. Moreover, the Poisson pressure equation is introduced, together with the corresponding Robin boundary conditions, and a suitable change of variable is employed. In section \ref{seipropriotu}, after after recalling some previous findings and known inequalities, we face with the proof of the existence and uniqueness of the nonlinear perturbed problem, taking advantage of several usual analytical techniques, e.g. derivation of \emph{a priori} estimates and use of the Galerkin method \cite{evans2022partial} with a suitable basis. Moreover, for \vir{sufficiently small} Rayleigh numbers all solutions must decay exponentially to zero as time increase, proving the nonlinear stability stability of the conduction solution.

\section{Formulation of the initial boundary value problem}\label{eu1}

Let us consider a reference frame $Oxyz$ with fundamental unit vectors $\{ \textbf{i},\textbf{j},\textbf{k}\}$ and a horizontal layer $\Omega=\mathbb{R}^2 \times [0,d]$ of fluid-saturated porous medium, whose boundary will be indicated by $ \partial \Omega=\Gamma=\Gamma_L\cup \Gamma_U $. To derive the governing equations for the seepage velocity $\textbf{v}$, the temperature field $T$ and the pressure field $p$, let us employ the modified Oberbeck-Boussinesq approximation, \cite{gouin2012consistent}: 
\begin{itemize}
	\item the fluid density $\rho$ is constant in all terms of the governing equations (i.e. $\rho=\rho_0$), except in the bouyancy term;
	\item in the body force term, the fluid density is \eqref{seitu};
	\item $\nabla \cdot \textbf{v}=0$ and $\textbf{D} : \textbf{D} \approx 0$.
\end{itemize}
Now, the mathematical model, according to the Darcy's law, is given by
\begin{equation}\label{syst1x}
	\begin{cases}
		\dfrac{\mu}{K} \textbf{v}= - \nabla p -\rho_0 [1-\alpha (T-T_0) + \beta (p-p_0)] g \textbf{k} \\ \nabla \cdot \textbf{v} = 0 \\ \rho c_V \Bigl( \dfrac{\partial T}{\partial t} + \textbf{v} \cdot \nabla T \Bigr) = \chi \Delta T
	\end{cases}
\end{equation}
where $\mu, K, c_a, \chi, c_V$ are fluid viscosity, permeability of the porous body, acceleration coefficient, thermal diffusivity and specific heat at constant volume, respectively. To system (\ref{syst1x}) the following initial conditions
\begin{equation}\label{IC1}
	\textbf{v}(\textbf{x},0)=\textbf{v}_0(\textbf{x}),\quad T(\textbf{x},0)=T_0(\textbf{x}),\quad p(\textbf{x},0)=p_0(\textbf{x}),
\end{equation}
and boundary conditions:
\begin{equation}\label{BC1x}
	\begin{aligned}
		\textbf{v} \cdot \textbf{k} =0 \quad & \text{on} \ z=0,d \\
		T=T_L  \quad & \text{on} \ z=0 \\
		T=T_U  \quad & \text{on} \ z=d 
	\end{aligned}
\end{equation}
with $T_L>T_U$, are appended. Conditions \eqref{BC1x}$ _1 $ tells us that the boundaries are impermeable, while condition \eqref{BC1x}$ _{2,3} $ that the boundaries are isothermal and we assume that the boundaries are isobaric too. Moreover, we can assume free boundary condition (i.e. boundary free from tangential stress).

\subsection{Steady state and perturbed non-dimensional formulation}

System (\ref{syst1x})-(\ref{BC1x}) admits the stationary conduction solution 
\begin{equation}
	\begin{aligned}
		&  \textbf{v}_b=\bm{0}, \ T_b(z)=T_L-\dfrac{T_L-T_U}{d} z, \\
		& p_b(z)=p_0+\bar{p}e^{-\rho_0g\beta z}+\dfrac{1}{\beta^2}\dfrac{\alpha(T_L-T_U)}{\rho_0 g d}\left(1-e^{-\rho_0g\beta z}\right)\\
		&\qquad\qquad\qquad\qquad\qquad\qquad\qquad-\dfrac{1}{\beta}\left(\dfrac{\alpha(T_L-T_U)}{d}z+1-e^{-\rho_0g\beta z}\right).
	\end{aligned}
\end{equation}
where $ p_0 $ is the gauge pressure and $ \bar{p} $ is a prescribed value arising from the appropriate boundary conditions on $ p $.\\
\\
Let $(\textbf{u},\theta,\pi)$ be a perturbation to the basic solution ($ \textbf{u}=\textbf{v}-\textbf{v}_b $, $ \theta=T-T_b $ and $ \pi=p-p_b $), with $ \textbf{u}=(u,v,w) $. Then, the equations governing the perturbation fields are
\begin{equation}\label{pert1x}
	\begin{cases}
		\dfrac{\mu}{K} \textbf{u} = - \nabla \pi + \rho_0 \alpha g \theta  \textbf{k} -  \rho_0 \beta g \pi  \textbf{k} \\ \nabla \cdot \textbf{u} = 0 \\ \dfrac{\partial \theta}{\partial t} + \textbf{u} \cdot \nabla \theta = \dfrac{T_L-T_U}{d} \textbf{u} \cdot \textbf{k} + k \Delta \theta
	\end{cases}
\end{equation}
where $k=\frac{\chi}{\rho c_V}$. To system (\ref{pert1x}) the following initial conditions
\begin{equation}\label{IC2}
	\textbf{u}(\textbf{x},0)=\textbf{u}_0(\textbf{x}),\quad \theta(\textbf{x},0)=\theta_0(\textbf{x}),\quad \pi(\textbf{x},0)=\pi_0(\textbf{x}),
\end{equation}
and boundary conditions are appended:
\begin{equation}\label{BC3}
	\begin{aligned}
		\textbf{u} \cdot \textbf{k} =0 \quad & \text{on} \ z=0,d \\
		\theta=0  \quad & \text{on} \ z=0,d.
	\end{aligned}
\end{equation}
Then system \eqref{syst1x} is fulfilled by the null perturbation. Let us introduce the following scales 
$$\pi=P \pi^*, \quad \textbf{u}=U \textbf{u}^*, \quad \theta=T^{\#} \theta^*, \quad t=\tau t^*, \quad x=dx^*, $$
where
\[ P=\dfrac{\mu k}{K},\quad U=\dfrac{k}{d},\quad T^\#=T_L-T_U, \quad \tau=\dfrac{d^2}{k}. \]
Therefore, the corresponding dimensionless system of equations, omitting all the stars, is the following:

\begin{equation}\label{nondimsysx}
	\begin{cases}
		\textbf{u}=-\nabla\pi+\Ra \theta \textbf{k}-\widehat{\beta}\pi \textbf{k}\\
		\nabla \cdot \textbf{u}=0\\
		\dfrac{\partial \theta}{\partial t}+\textbf{u}\cdot \nabla \theta=w+\Delta \theta
	\end{cases}
\end{equation}
where
\[ \Ra=\dfrac{\rho_0\alpha g d(T_L-T_U)K}{\mu k}, \quad \widehat{\beta}=\rho_0gd\beta  \]
are the Darcy-Rayleigh number and the dimensionless compressibility factor, respectively. To system \eqref{nondimsysx} the following boundary conditions are appended:

\begin{equation}\label{nondimbcx}
	w=\theta=0\quad \text{on}\; z=0,1.
\end{equation}

\subsection{The Poisson pressure equation}

It is well known that the pressure field in case of incompressible flow can be recognized as the Lagrangian constraint variable that enforces the divergence-free constraint. In this new scheme, the well-posedness of the associated mathematical problem requires now the pressure field to be treated as an \emph{independent unknown}, satisfying a suitable elliptic problem and subject to Robin boundary conditions. Indeed, by taking the divergence of \eqref{nondimsysx}$ _1 $ we have
\[ \nabla \cdot \textbf{u}=-\nabla \cdot \nabla\pi+\Ra \nabla \cdot \theta \textbf{k}-\widehat{\beta}\nabla \cdot \pi \textbf{k}, \]
hence
\begin{equation}\label{PPE}
	\Delta \pi+\widehat{\beta}\dfrac{\partial \pi}{\partial z}=\Ra \dfrac{\partial \theta}{\partial z},
\end{equation}
whose boundary conditions are, following \cite{sani2006pressure}
\[ \dfrac{\partial p}{\partial \textbf{n}}\Bigl|_{\Gamma}=\nabla \pi\cdot \textbf{n}=\nabla \pi\cdot \textbf{k}=(-\textbf{u}+\Ra\theta\textbf{k}-\widehat{\beta}\pi\textbf{k})\cdot \textbf{k}\Bigl|_{\Gamma}, \]
hence the \emph{natural} boundary conditions for the pressure Poisson problem \eqref{PPE} are of Robin type
\begin{equation}\label{RBC}
	\dfrac{\partial \pi}{\partial z}+\widehat{\beta}\pi=0.
\end{equation}
Therefore, we replace the continuity constraint \eqref{nondimsysx} with the derived pressure Poisson equation \eqref{PPE}, so we obtain
\begin{equation}\label{mainsystem}
	\begin{cases}
		\textbf{u}=-\nabla\pi+\Ra \theta \textbf{k}-\widehat{\beta}\pi \textbf{k}\\
		\Delta \pi+\widehat{\beta}\dfrac{\partial \pi}{\partial z}=\Ra \dfrac{\partial \theta}{\partial z}\\
		\dfrac{\partial \theta}{\partial t}+\textbf{u}\cdot \nabla \theta=w+\Delta \theta
	\end{cases}
\end{equation}
with associated boundary conditions
\begin{equation}\label{key}
	w=\theta=\dfrac{\partial \pi}{\partial z}+\widehat{\beta}\pi=0\quad \text{on}\; z=0,1.
\end{equation}
If we consider the change of variable $ \Pi=e^{\widehat{\beta}z}\pi $, then system \eqref{mainsystem} turns into
\begin{equation}\label{mainsystemchange}
	\begin{cases}
		\textbf{u}=-e^{-\widehat{\beta}z}\nabla\Pi+\Ra \theta \textbf{k}\\
		\Delta \Pi-\widehat{\beta}\dfrac{\partial \Pi}{\partial z}=\Ra e^{-\widehat{\beta}z} \dfrac{\partial \theta}{\partial z}\\
		\dfrac{\partial \theta}{\partial t}+\textbf{u}\cdot \nabla \theta=w+\Delta \theta
	\end{cases}
\end{equation}
with initial conditions
\begin{equation}\label{ICx}
	\textbf{u}(\textbf{x},0)=\textbf{u}_0(\textbf{x}),\quad \theta(\textbf{x},0)=\theta_0(\textbf{x}),\quad \Pi(\textbf{x},0)=\Pi_0(\textbf{x}),	
\end{equation}
and the Robin condition for the pressure turned into Neumann conditions
\begin{equation}\label{NBC}
	w=\theta=\dfrac{\partial \Pi}{\partial z}=0\quad \text{on}\; z=0,1.
\end{equation}
We are now in a full coupling contest and note that $ t $ is just a parameter for the elliptic problem \eqref{mainsystemchange}$ _2 $ and \eqref{NBC}. \\
\\
In the sequel we will focus on bi-dimensional perturbations in the plane $ (x,z) $ and we assume the periodicity of the perturbation in the $ x $-direction.

\section{Well-posedness of the problem}\label{seipropriotu}

Let us consider the full-nonlinear system

\begin{equation}\label{system2ABC}
	\begin{cases}
		\Delta \Pi-\widehat{\beta}\Pi_z=\mathsf{Ra} e^{\widehat{\beta}z} \theta_z\\
		\textbf{u}=-e^{-\widehat{\beta}z}\nabla\Pi+\Ra\theta \textbf{k}\\
		\theta_t+\textbf{v}\cdot \nabla \theta=w+\Delta \theta
	\end{cases}
\end{equation}
together with boundary conditions \eqref{NBC}. In the following sections we will give some preliminary notation and preliminaries in order step by step prove the existence and uniqueness of weak solutions of \eqref{system2ABC}. 

\subsection{Some notation and preliminaries}

In the following, since we are interested in bi-dimensional flow, we will denote
\begin{equation}\notag
	\Omega_0=\{(x,y)\in [0,1]\times [0,1]\}.
\end{equation}
Let us consider $ \mathcal{B}=\{\phi^i_{m,n}\}_{m,n\geq 0} $ such that
\begin{equation}\label{baseB}
	\phi^i_{m,n}(x,z)=\begin{cases}
		\cos(2\pi mx)\cos(\pi nz),\quad i=1,\\
		\sin(2\pi mx)\cos(\pi nz),\quad i=-1,
	\end{cases}
\end{equation}
and $ \mathcal{D}=\{\xi^i_{m,n}\}_{m,n\geq 0} $ such that
\begin{equation}\label{baseD}
	\xi^i_{m,n}(x,z)=\begin{cases}
		\cos(2\pi mx)\sin(\pi nz),\quad i=1,\\
		\sin(2\pi mx)\sin(\pi nz),\quad i=-1.
	\end{cases}	
\end{equation}
Note that the functions in the basis $ \mathcal{B} $ have mean value zero, so we can write $ p $ in this basis. Moreover, for temperature field $ \theta $ and the stream function $ \Phi $ associated with $ \textbf{v} $ by $ \textbf{u}=(-\Phi_z,\Phi_x) $ we can use the basis $ \mathcal{D} $. We can construct Sobolev space from basis \eqref{baseB} and \eqref{baseD}. In particular, we denote by $ \widetilde{W}^{k,2}(\Omega_0) $ as the closure with respect to the $ W^{k,2} $-norm of the finite combinations of elements of the basis $ \mathcal{B} $ and by $ \widehat{W}^{k,2}(\Omega_0) $ as the closure with respect to the $ W^{k,2} $-norm of the finite combinations of elements of the basis $ \mathcal{D} $, $ k=1,2 $. Moreover, we denote by $ \widehat{\mathcal{W}}^{1,2}(\Omega_0) $ the closure of the linear hull of the vectorial divergence-free functions obtained from $ \mathcal{D} $. Let us recall the Poincaré inequality for $ \Pi $ and $ \theta $:
\begin{equation}\notag
	\|\Pi\|_{L^2}\leq \dfrac{1}{2\pi}\|\nabla\Pi\|_{L^2},\quad \|\theta\|_{L^2}\leq \dfrac{1}{\sqrt{5}\pi}\|\nabla\theta\|_{L^2},
\end{equation}
as well as the Ladyzhenskaya's inequality
\begin{equation}\notag
	\|\textbf{v}\|_{L^4}\leq c \|\textbf{v}\|_{L^2}^{1/2}\|\nabla\textbf{v}\|_{L^2}^{1/2}
\end{equation}
with $ c $ positive constant, and the following equivalence of norms, \cite{de2019existence}
\begin{equation}\notag
	\dfrac{1}{16}\|\Delta \textbf{v}\|_{L^2}\leq \|D^2\textbf{v}\|_{L^2}\leq \dfrac{1}{4}\|\Delta \textbf{v}\|_{L^2}.
\end{equation}

\subsection{Existence and uniqueness for the reduced problem}

As a fist result, let us recall the following theorem proving existence and uniqueness for the reduced problem for the pressure. 

\begin{theorem}\label{reducedproblem}
	Let $ f\in \widetilde{L}^2(\Omega_0) $ and assume $ 0\leq \widehat{\beta}<2\pi $. Then problem
	\begin{equation}\label{redprob}
		\begin{cases}
			\Delta \Pi-\widehat{\beta}\Pi_z=e^{\widehat{\beta}z}f,\quad \text{in}\;\Omega\\
			\Pi_z(x,0)=\Pi_z(x,1)=0,\quad \text{in}\;x\in\mathbb{R},
		\end{cases} 
	\end{equation}
	admits a unique solution $ \Pi\in \widetilde{W}^{2,2}(\Omega_0) $. Moreover, $ \Pi $ satisfies the following estimates
	\begin{equation}\label{estimeciao}
		\|\nabla\Pi\|_{L^2}\leq \dfrac{1}{2\pi-\widehat{\beta}}\left\|e^{\widehat{\beta}z}f\right\|_{L^2},\quad \|\Delta\Pi\|_{L^2}\leq \dfrac{2\pi}{2\pi-\widehat{\beta}}\left\|e^{\widehat{\beta}z}f\right\|_{L^2}. 
	\end{equation}
\end{theorem}
\begin{proof}
	A proof of the following theorem can be found in \cite{de2019existence}. 
\end{proof}

\subsection{Basic \emph{a priori} estimate}

In this section we will prove a fundamental \emph{a priori} estimate. Let us set

\begin{equation}\label{key}
	E(t):=\dfrac{\Ra}{2}\dfrac{d}{dt}\|\theta(t)\|^2_{L^2}.
\end{equation}
Then the following proposition holds:

\begin{prop}\label{proposizione1}
	The following \emph{a priori} estimate holds:
	\begin{equation}\label{key}
		\dfrac{\Ra}{2}\dfrac{d}{dt}\|\theta\|^2+\Ra\|\nabla\theta\|^2_{L^2}-A\|\textbf{u}\|^2_{L^2}-B\|\nabla\Pi\|^2_{L^2}\leq c_0(\Ra,\widehat{\beta})\|\theta\|^2_{L^2},
	\end{equation}
	where $ c_0>0 $, $ A,B<0 $, and in particular 
	\begin{equation}\label{davide}
		E(t)\leq E(0)e^{c_0 t}. 
	\end{equation}
	Moreover, if $ \Ra $ is sufficiently small, then $ E(t) $ decays exponentially. 
\end{prop}
\begin{proof}
Multiplying \eqref{system2ABC}$ _1 $ by $ \Pi $, \eqref{system2ABC}$ _2 $ by $ \textbf{u} $ and \eqref{system2ABC}$ _3 $ by $ \Ra\theta $, after integrating over the periodicity cell, applying Cauchy-Schwartz and generalized Young inequalities and summing the resulting equations we get
	\begin{equation}\label{mattia}
		\begin{split}
			\dfrac{\Ra}{2}\dfrac{d}{dt}\|\theta\|^2&=-\|\nabla\Pi\|^2_{L^2}-\widehat{\beta}(\Pi_z,\Pi)-\Ra(e^{\widehat{\beta}z}\theta_z,\Pi)\\
			&\quad\quad\quad\qquad\qquad -\|\textbf{u}\|^2+2\Ra(\theta,w)-(e^{-\widehat{\beta}z}\nabla\Pi,\textbf{u})-\Ra\|\nabla\theta\|^2_{L^2},
		\end{split}
	\end{equation}
and it follows from \eqref{mattia} that
	\begin{equation}\label{francis}
		\begin{split}
			\dfrac{\Ra}{2}\dfrac{d}{dt}\|\theta\|^2+\Ra\|\nabla\theta\|^2_{L^2}-A\|\textbf{u}\|^2_{L^2}-B\|\nabla\Pi\|^2_{L^2}\leq c_0(\Ra,\widehat{\beta})\|\theta\|^2_{L^2}	
		\end{split}
	\end{equation}
	with
	\begin{equation}\label{key}
		\begin{split}
			A:&=-1+\dfrac{1}{2M_1}+\dfrac{M_2}{2},\\
			B:&=-1+\dfrac{\widehat{\beta}}{2\pi}+\dfrac{1}{2M}\left(\dfrac{1}{2\pi}+1\right)+\dfrac{1}{2M_2},\\
			c_0:&=\dfrac{\Ra^2e^{2\widehat{\beta}}}{2}(\widehat{\beta}^2+1)M+2\Ra^2M_1.
		\end{split}
	\end{equation}
	If we ask $ A<0 $ and $ B<0 $, the following restrictions hold
	\begin{equation}\label{key}
		\begin{split}
			M_1&>\dfrac{1}{2-M_2},\\
			M_2&\in \left(\dfrac{\pi}{2\pi-\widehat{\beta}},2\right),\\
			M&>\dfrac{(2\pi+1)M_2}{2(2\pi-\widehat{\beta}-2\pi)}.
		\end{split}
	\end{equation}
	As a consequence, from \eqref{francis}, we have
	\begin{equation}\label{key}
		\dfrac{d}{dt}E(t)\leq c_0(\Ra,\widehat{\beta})E(t).
	\end{equation}
	and the Gr\"onwall inequality gives us \eqref{davide}. Moreover, setting
	\begin{equation}\label{key}
		\hat{c}_0(\Ra,\widehat{\beta}):=\Ra\left(\dfrac{\Ra e^{2\widehat{\beta}}}{2}(\widehat{\beta}^2+1)M+2\Ra M_1\right),
	\end{equation}
	from \eqref{francis} it follows
	\begin{equation}\label{key}
		\begin{split}
			\dfrac{\Ra}{2}\dfrac{d}{dt}\|\theta(t)\|^2&\leq -\Ra \|\nabla\theta\|^2_{L^2}+c_0(\widehat{\beta},\Ra)\|\theta\|^2_{L^2}\\
			&\leq -10\pi^2 \dfrac{\Ra}{2}\|\theta\|^2_{L^2}+2\hat{c}_0(\Ra,\widehat{\beta})\dfrac{\Ra}{2}\|\theta\|^2_{L^2}\\
			&=\left( -\widetilde{c}_1+\widetilde{c}_0 \right)\|\theta\|^2_{L^2},
		\end{split}
	\end{equation}
	where $ \widetilde{c}_1=10\pi^2 $ and $ \widetilde{c}_0=2\hat{c}_0 $, i.e.
	\begin{equation}\label{key}
		\dfrac{d}{dt}E(t)\leq (\widetilde{c}_0-\widetilde{c}_1)E(t).
	\end{equation}
	In conclusion, from a Gr\"onwall type inequality we have
	\begin{equation}\label{key}
		E(t)\leq E(0)^{ (\widetilde{c}_0-\widetilde{c}_1)t},
	\end{equation}
	and if $ \Ra $ is sufficiently small so that $ \widetilde{c}_0-\widetilde{c}_1<0 $, i.e. if
	\begin{equation}\label{key}
		\Ra<\dfrac{10\pi^2}{e^{2\widehat{\beta}}(\widehat{\beta}^2+1)+4M_1},
	\end{equation}
	then $ E $ decays exponentially, and this completes the proof. 
\end{proof}

\subsection{Preliminary results}

In this section we will give several preliminary results, which will make the final proof easier to demonstrate. 

\begin{lem}\label{lemma1}
	The sequence $ \{\theta^N\}_{N} $ is bounded in $ L^\infty\bigl(0,T;\widetilde{L}^2(\Omega_0)\bigr) $, i.e. it exists $ c_1>0 $ such that for all $ N $
	\[ \esssup_t\|\theta^N(t)\|^2_{L^2}\leq c_1. \]
\end{lem}
\begin{proof}
	From the \emph{a priori} estimates given in Proposition \ref{proposizione1}, it follows
	\begin{equation}\label{key}
		E^N(t)\leq E^N(0)e^{c_0t},
	\end{equation}
	hence
	\begin{equation}\label{key}
		\begin{split}
			\dfrac{\Ra}{2}\|\theta^N(t)\|^2_{L^2}\leq \dfrac{\Ra}{2}\|\theta^N(0)\|^2_{L^2}e^{c_0t}\leq \dfrac{\Ra}{2}\|\theta(0)\|^2_{L^2}e^{c_0T}=c_1.
		\end{split}
	\end{equation}
	Then, passing to the $ \esssup $ the Lemma is proved. 
\end{proof}

\begin{lem}\label{lemma2}
	The sequence $ \{\theta^N\}_{N} $ is bounded in $ L^2\bigl(0,T;\widetilde{L}^2(\Omega_0)\bigr) $
\end{lem}
\begin{proof}
	It directly follows from Lemma \ref{lemma1} that
	\[ \|\theta^N\|^2_{L^2(0,T;\widetilde{L}^2(\Omega_0))}=\int_0^T \|\theta^N(t)\|^2_{L^2}dt\leq T c_1. \]
\end{proof}

\begin{lem}\label{lemma3}
	The sequence $ \{\nabla\theta^N\}_{N} $ is bounded in $ L^2\bigl(0,T;\widetilde{L}^2(\Omega_0)\bigr) $
\end{lem}
\begin{proof}
	From Proposition \ref{proposizione1}, we have
	\begin{equation}\label{key}
		\dfrac{\Ra}{2}\dot{E}^N(t)+\Ra\|\nabla\theta^N(t)\|^2_{L^2}\leq c_0(\Ra,\widehat{\beta})\|\theta^N(t)\|^2_{L^2}.
	\end{equation}
	As a consequence
	\begin{equation}\label{key}
		\int_0^T\Ra \|\nabla\theta^N(t)\|_{L^2}^2dt\leq \int_0^Tc_0\|\theta^N(t)\|_{L^2}^2dt-\dfrac{\Ra}{2}\int_0^T\dot{E}^N(t)\,dt
	\end{equation}
	then, from Lemma \ref{lemma2}, it follows
	\begin{equation}\label{key}
		\begin{split}
			\Ra\|\nabla\theta^N\|_{L^2(0,T;\widetilde{L}^2(\Omega_0))}^2&\leq Tc_0c_1-\dfrac{\Ra}{2}\bigl(E^n(t)-E^n(0)\bigr)\\
			&\leq T c_0c_1+\dfrac{\Ra}{2}E^N(0)\\
			&\leq \Ra\left(\dfrac{T}{\Ra}c_0c_1+\dfrac{1}{2}\|\theta(0)\|^2_{L^2}\right)\\
			&=\Ra c_2
		\end{split}
	\end{equation}
	i.e.
	\begin{equation}\label{key}
		\|\nabla\theta^N\|_{L^2(0,T;\widetilde{L}^2(\Omega_0))}^2=c_2.
	\end{equation}
\end{proof}

\begin{lem}\label{lemma4}
	Let $ \gamma>0 $. Then, the following estimate holds:
	\begin{equation}\label{key}
		\|w^N(t)\|^2_{L^2}\leq \gamma \|\theta^N(t)\|^2_{W^{1,2}}.
	\end{equation}
\end{lem}
\begin{proof}
	It follows directly from Theorem \ref{reducedproblem}, setting $ \gamma=\max\left\{ \tfrac{e^{2\widehat{\beta}}}{2\pi -\widehat{\beta}},\Ra \right\} $, that
	\begin{equation}\label{key}
		\begin{split}
			\|w^N(t)\|^2_{L^2}&\leq \|\nabla\Pi^N\|^2_{L^2}+\Ra\|\theta^N\|^2_{L^2}\\
			&\leq \tfrac{1}{2\pi -\widehat{\beta}}\bigl\|e^{\widehat{\beta}z}\theta_z^2(t)\bigr\|^2_{L^2}+\Ra\|\theta^N(t)\|^2_{L^2}\\
			&\leq \tfrac{e^{2\widehat{\beta}}}{2\pi -\widehat{\beta}}\bigl\|\nabla\theta^2(t)\bigr\|^2_{L^2}+\Ra\|\theta^N(t)\|^2_{L^2}\\
			&\leq \gamma \|\theta^N(t)\|^2_{W^{1,2}}.
		\end{split}
	\end{equation}
\end{proof}

\begin{lem}\label{lemma5}
	The sequence $ \{\dot{\theta}^N\}_{N} $ is bounded in $ L^2\bigl(0,T;(\widetilde{W}^{1,2}(\Omega_0))^*\bigr) $
\end{lem}
\begin{proof}
	Let $ v\in \widetilde{W}^{1,2}(\Omega_0) $, such that $ v=v_1+v_2 $ with $ v_2\perp \{\phi^i_{mn},\;|\mu|<N\}=\bar{V}_N $ and $ v_1\perp v_2 $. Then from Lemma \ref{lemma4} and Lemma \ref{lemma2}, \ref{lemma3} it follows that
	\begin{equation}\label{key}
		\begin{split}
			(\dot{\theta}^N,v)&=(\dot{\theta}^N,v_1)\\
			&=(w^N,v_1)+(\Delta \theta^N,v_1)\\
			&=(w^N,v_1)-(\nabla\theta^N,\nabla v_1)\\
			&\leq \|w^N\|_{L^2}\|v_1\|_{L^2}+\|\nabla\theta^N\|_{L^2}\|\nabla v_1\|\\
			&\leq \gamma' \|\theta^N\|_{W^{1,2}}\|v_1\|_{L^2}+\|\nabla\theta^N\|_{L^2}\|\nabla v_1\|\\
			&\leq c_3\|v_1\|_{W^{1,2}}.
		\end{split}
	\end{equation}
	Therefore, we have
	\begin{equation}\label{key}
		\dfrac{(\dot{\theta}^N,v)}{\|v_1\|_{W^{1,2}}}\leq c_3,
	\end{equation}
	and passing to the $ \sup $ we obtain
	\begin{equation}\label{key}
		\|\dot{\theta}^N(t)\|_{(\widetilde{W}^{1,2}(\Omega_0))^*}=\sup \dfrac{(\dot{\theta}^N,v)}{\|v_1\|_{W^{1,2}}}\leq c_3.
	\end{equation}
	Finally, the thesis follows integrating the previous equation, i.e.
	\begin{equation}\label{key}
		\|\dot{\theta}^N\|_{L^2(0,T;(\widetilde{W}^{1,2}(\Omega_0))^*)}=\int_0^T\|\dot{\theta}^N(t)\|_{(\widetilde{W}^{1,2}(\Omega_0))^*}dt\leq c_3T.
	\end{equation}
\end{proof}
\noindent
We are now ready to prove the first main preliminary result.
\begin{prop}\label{proposizione2}
	The sequences $ \{\theta^N\}_N $ and $ \{\dot{\theta}^N\}_N $ are weakly relatively sequentially compact. 
\end{prop}
\begin{proof}
	First of all, it follows from Lemma \ref{lemma2} that $ \{\theta^N\}_N $ is bounded in $ L^2(0,T;\widetilde{W}^{1,2}(\Omega_0)) $ and from Lemma \ref{lemma5} that $ \{\dot{\theta}^N\}_N $ is bounded in the space $ L^2(0,T;(\widetilde{W}^{1,2}(\Omega_0))^*) $. Hence, the Banach-Alaoglu theorem guarantees that $ \{\theta^N\}_N $ and $ \{\dot{\theta}^N\}_N $ are weakly-* compact in $ L^2(0,T;\widetilde{W}^{1,2}(\Omega_0)) $ (which is Hilbert). Therefore, the two sequences are weakly compact, i.e. there exist two subsequences $ \{\theta^{N_j}\}_j $ and $ \{\dot{\theta}^{N_j}\}_j $ such that
	\begin{equation}\label{key}
		\theta^{N_j}\rightharpoonup \theta\;\;\text{and}\;\; \dot{\theta}^{N_j}\rightharpoonup \psi.
	\end{equation}
\end{proof}
\noindent
\begin{cor}
	It results $ \psi=\dot{\theta} $. 
\end{cor}
\begin{proof}
	Let us consider $ f(t)\in C^\infty_c(0,T) $ and $ w\in \widetilde{W}^{1,2}(\Omega_0) $, so the product $ fw\in  L^2\bigl(0,T;\widetilde{W}^{1,2}(\Omega_0)\bigr) $. It follows from the Tonelli-Fubini theorem that
	\begin{equation}\label{key}
		\begin{split}
			A^N:&=\int_0^T\left(\int_{\Omega_0}\dot{\theta}^N(t)f(t)w\,d\Omega\right)dt=\int_{\Omega_0}\left(\int_0^T\dot{\theta}^N(t)f(t)w\,dt\right)d\Omega\\
			&=-\int_{\Omega_0}\left(\int_0^T{\theta}^N(t)\dot{f}(t)w\,dt\right)d\Omega=-\int_0^T\left(\int_{\Omega_0}{\theta}^N(t)\dot{f}(t)w\,d\Omega\right)dt\\
			&=:B^N.
		\end{split}
	\end{equation}
	Hence
	\begin{equation}\label{key}
		\lim_NA^N=\int_0^T\left\langle \psi(t),f(t)w \right\rangle=-\int_0^T\left(\theta(t),\dot{f}(t)w\right)dt=\lim_NB^N,
	\end{equation}
	i.e., by definition
	\begin{equation}\label{key}
		\psi=\dot{\theta}. 
	\end{equation}
\end{proof}
\noindent
Let us prove now some other preliminary Lemmas.
\begin{lem}\label{lemma6}
	The sequence $ \{\textbf{u}^N\}_N $ is bounded in $ L^2\bigl(0,T;\widehat{L}^2(\Omega_0)\bigr) $.
\end{lem}
\begin{proof}
	It is easy to see that
	\begin{equation}\label{key}
		\begin{split}
			\|\textbf{u}^N(t)\|_{L^2}^2&\leq \|\nabla\Pi^N\|_{L^2}^2+\Ra^2\|\theta^N\|_{L^2}^2\\
			&\leq c\|\theta_z^N(t)\|_{L^2}^2+\Ra\|\theta^N(t)\|_{L^2}^2\\
			&\leq \hat{c} \|\theta^N(t)\|^2_{W^{1,2}}
		\end{split}
	\end{equation}
	and integrating, recalling Lemma \ref{lemma2}, \ref{lemma3} we obtain
	\begin{equation}\label{key}
		\int_0^T\|\textbf{u}^N(t)\|_{L^2}^2dt\leq \hat{c}\int_0^T\|\theta^N(t)\|^2_{W^{1,2}}dt\leq  c_4. 
	\end{equation}
\end{proof}

\begin{lem}\label{lemma7}
	The sequence $ \{\nabla\Pi^N\}_N $ is bounded in $ L^2\bigl(0,T;\widetilde{L}^2(\Omega_0)\bigr) $.
\end{lem}
\begin{proof}
	The proof immediately follows from the fact that
	\begin{equation}\label{key}
		\|\nabla\Pi^N(t)\|^2_{L^2}\leq c\|\nabla\theta^N(t)\|^2_{L^2},
	\end{equation}
	and after integrating and exploiting Lemma \ref{lemma3}.
\end{proof}

\begin{cor}
	The sequence $ \{\Pi^N\}_N $ is bounded in $ L^2\bigl(0,T;\widetilde{L}^2(\Omega_0)\bigr) $.
\end{cor}
\begin{proof}
	Since $ \langle \Pi \rangle =0 $, then the proof follows from Lemma \ref{lemma7} and the Poincaré inequality. 
\end{proof}
We are now ready to prove the second main preliminary result.
\begin{prop}\label{proposizione3}
	The sequences $ \{\textbf{u}^N\}_N $ and $ \{\Pi^N\}_N $ are weakly relatively sequentially compact. 
\end{prop}

\begin{proof}
	The result can be proved analogously as in Proposition \ref{proposizione2}. Indeed, we obtain that there exists two subsequences
	\begin{equation}\label{key}
		\textbf{u}^{N_j}\rightharpoonup \textbf{u}\;\;\text{and}\;\; {\Pi}^{N_j}\rightharpoonup \Pi,
	\end{equation}
	weakly convergent in $ L^2\bigl(0,T;\widehat{L}^2(\Omega_0)\bigr) $ and $ L^2\bigl(0,T;\widetilde{W}^{1,2}(\Omega_0)\bigr) $ respectively. 
\end{proof}
\noindent
The following result concern the existence and uniqueness os solutions for the linear version of \eqref{system2ABC}. 
\begin{prop}\label{proposizione4}
	$ (\theta, \textbf{u},\Pi) $ is a weak solution of the linear version of system \eqref{system2ABC}.
\end{prop}
\begin{proof}
	Let us consider
	\begin{equation}
		v=\sum_{|\mu| \leq M}v_{mn}^i(t)\phi_{mn}^i \in \widetilde{W}^{1,2}(\Omega_0).
	\end{equation}
	Let $ N>M $, then we have that
	\begin{equation}
		\int_{\Omega_0}\dot{\theta}^N(t)\phi^i_{mn}d\Omega=\int_{\Omega_0}w^N \phi_{mn}^id\Omega-\int_{\Omega_0}\nabla \theta^N\nabla \phi_{mn}^id\Omega	
	\end{equation}
	and integrating we obtain
	\begin{equation}
		\begin{split}
			\int_0^T\left( \int_{\Omega_0} \dot{\theta}^N(t)v(t) \right)dt=&\int_0^T \left( \int_{\Omega_0}w^n (t)v(t) d\Omega \right)dt-	\\&\qquad\qquad	\int_0^T\left(\int_{\Omega_0} \nabla \theta^N(t)\cdot \nabla v(t)d\Omega \right)dt.
		\end{split}
	\end{equation}
	By Proposition \ref{proposizione1}-\ref{proposizione2}, it follow that 
	\begin{equation}
		\begin{split}
			\lim_N\langle \dot{\theta}^N,v \rangle &= \lim_N \int_0^T\left(\int_{\Omega_0} \dot{\theta}^N(t) v(t) d\Omega \right)dt=\langle\dot{\theta},v\rangle, \\ 	
			\lim_N(w^N,v)&=\lim_N \int_0^T\left(\int_{\Omega_0} w^N(t) v(t) d\Omega \right)dt=(w,v),\\
			\lim_N(\nabla\theta^N,\nabla v)&=\lim_N \int_0^T\left(\int_{\Omega_0} \nabla\theta^N\cdot \nabla v d\Omega \right)dt=(\nabla\theta,\nabla v).
		\end{split} 
	\end{equation}
	Therefore, $\forall M, \forall v\in L^2(0,T,\widetilde{W}^{1,2}(\Omega_0))$ we have that
	\begin{equation}\label{eqdue}
		\langle\dot{\theta},v\rangle=(w,v)-(\nabla \theta,\nabla v)
	\end{equation}
	Let us choose $v=\varphi g$, $\varphi \in C_c^\infty(0,T)$, $g\in\widetilde{W}^{1,2}(\Omega_0)$, hence \eqref{eqdue} becomes
	\begin{equation}
		\langle\dot{\theta}, \varphi g\rangle=(w,\varphi g)-(\nabla \theta, \nabla(\varphi g))=(w,\varphi g)-(\nabla \theta, \varphi\nabla g)
	\end{equation}
	If we integrate we obtain
	\begin{equation}
		\int_0^T\langle\dot{\theta}(t), \varphi(t) g\rangle dt=\int_0^T(w(t),\varphi(t) g)dt-\int_0^T(\nabla \theta(t), \varphi(t)\nabla g)dt
	\end{equation}
	i.e.
	\begin{equation}
		\int_0^T \left[ \langle\dot{\theta}(t), g\rangle -(w(t),g)+(\nabla \theta(t),\nabla g) \right]\varphi (t)dt=0
	\end{equation}
	It follows from the Fundamental Lemma of Calculus of Variations that
	\begin{equation}\label{eqquattro}
		\langle\dot{\theta}, g\rangle -(w(t),g)+(\nabla \theta(t),\nabla g) =0,
	\end{equation}
	for a.e. $ t\in (0,T) $ and for all $ g\in \widetilde{W}^{1,2}(\Omega_0) $.\\
	We can apply the same procedure to \eqref{system2ABC}$ _1 $ and \eqref{system2ABC}$ _2 $. Now, recalling that
	$\theta\in L^2(0,T,\widetilde{W}^{1,2}(\Omega_0))$ and $\dot{\theta}\in L^2(0,T,(\widetilde{W}^{1,2}(\Omega_0))^*)$ it follows that, \cite[Thm, 5.9.3]{evans2022partial}
	\begin{equation}\label{key}
		\theta \in C(\left[ 0,T \right],\widetilde{L}^2(\Omega_0)),
	\end{equation}	
	At this point, we have that exists $ \theta(\text{\textbf{x}},0)$. We need now to prove that it is indeed $\theta_0$. To this aim, let us consider $v \in C^1(\left[0,T\right],\widetilde{W}^{1,2}(\Omega_0))$ such that $v(T)=0$. Hence
	\begin{equation}
		\int_0^T \langle\dot{\theta}^N,v\rangle dt=\int_0^T (w^N(t),v(t))dt-\int_0^T (\nabla \theta^N(t),\nabla v(t))dt,
	\end{equation}
	so
	\begin{equation}
		\begin{split}
			&-\int_0^T \left(\int_{\Omega_0} \theta^N(t,x)\dot{v}(t,x)d\Omega\right)dt-\int_{\Omega_0}\theta^N(0)v(0)d\Omega\\
			&\qquad\qquad\qquad\qquad=\int_0^T (w^N(t),v(t))dt-\int_0^T (\nabla \theta^N(t), \nabla v(t)).		
		\end{split}
	\end{equation}
	Then, if we take the limit we obtain
	\begin{equation}\label{equno}
		\begin{split}
			&\lim_N\left(-\int_0^T\left(\int_{\Omega_0}\theta^N(t)\dot{v}(t)d\Omega\right) dt\right)= -\int_0^T (\theta(t), \dot{v}(t)) -(\theta_0,v(0))\\&\qquad\qquad\qquad\qquad= \int_0^T (w(t),v(t))dt-\int_0^T (\nabla \theta(t), \nabla v(t)).
		\end{split}
	\end{equation}
	Moreover, from \eqref{eqdue}, we have $ \langle\dot{\theta},v\rangle=(w,v)-(\nabla \theta, \nabla v) $, but since
	\begin{equation}\label{eqtre}
		(\theta, \dot{v})-(\theta(0),\dot{v})-(\theta(0),v(0))=(w,v)-(\nabla \theta, \nabla v),
	\end{equation}
	from \eqref{equno} and \eqref{eqtre}, we have 
	\begin{equation}
		(\theta(0)-\theta_0, v_0)=0,
	\end{equation}
	i.e. 
	\begin{equation}\label{inizvox}
		\theta(0)=\theta_0 .
	\end{equation}
	Equation \eqref{eqquattro} together with \eqref{inizvox} $\theta$ tells us that $ \theta $ is a weak solution of \eqref{system2ABC}$ _3 $. We can apply the same procedure to equations \eqref{system2ABC}$ _1 $ and $ _2 $.\\
	Concerning the uniqueness, let us suppose that $ (\theta_1,\textbf{u}_1,\pi_1) $ and $ (\theta_2, \textbf{u}_2, \pi_2)$ are two solutions to system \eqref{system2ABC} such that
	\begin{equation}
		\theta=\theta_1-\theta_2, \quad \textbf{u}=\textbf{u}_1-\textbf{u}_2, \quad \pi=\pi_1-\pi_2.
	\end{equation}
	Since $(\theta, \textbf{u}, \pi)$ solve the system, with zero initial data, from \eqref{davide} we have that 
	\begin{equation}
		E(t)=\frac{\Ra}{2}\|\theta(t)\|_{L^2}^2\leq \frac{\Ra}{2}\|\theta_0\|^2_{L^2}e^{c_0t}=0.
	\end{equation}
	Therefore, the following chain of equivalence holds:
	\begin{equation}
		\|\theta(t)\|^2_{L^2}=0\;\; \Longleftrightarrow\;\; \theta(t)=0\;\; \Longleftrightarrow\;\; \theta_1-\theta_2=0 \;\;\Longleftrightarrow\;\; \theta_1=\theta_2.
	\end{equation}
	Form Theorem \ref{reducedproblem} we immediately have that $\Pi_1=\Pi_2$ and finally
	\begin{equation}
		\textbf{u}_1=-e^{-\hat{\beta} z}\nabla \Pi_1+R\theta_1 \textbf{k}=-e^{-\hat{\beta}z }\nabla \Pi_2+R \theta_2 \textbf{k}=\textbf{u}_2.
	\end{equation}
\end{proof}

\subsection{The main theorem}

The following preliminary results are needed in order to prove the existence and uniqueness of solutions of the full nonlinear system \eqref{system2ABC}. 

\begin{lem}\label{lemma8}
	The following estimate holds
	\begin{equation}\label{key}
		\|\nabla\textbf{u}\|_{L^2}\leq (c(\widehat{\beta})+\Ra)\|\nabla\theta\|_{L^2}
	\end{equation}	
\end{lem}
\begin{proof}
	Let us consider the divergence of \eqref{system2ABC}$ _{2} $, i.e.
	\begin{equation}\label{key}
		\nabla \textbf{u}=\nabla \left(-e^{-\widehat{\beta}z}\nabla\Pi\right)+\nabla (\Ra\theta\textbf{k}),
	\end{equation}
	in particular
	\begin{equation}\label{key}
		\nabla\textbf{u}=\left(\begin{array}{c}
			0\\
			\widehat{\beta}e^{-\widehat{\beta}z}\nabla\Pi
		\end{array}\right)-e^{-\widehat{\beta}z}D^2\Pi+\Ra\left(\begin{array}{c}
			0\\
			\nabla\theta
		\end{array}\right).
	\end{equation}
	We can then estimate as follows
	\begin{equation}\label{key}
		\begin{split}
			\|\nabla\textbf{u}\|_{L^2}&\leq \widehat{\beta}\|\nabla\Pi\|_{L^2}+e^{\widehat{\beta}}\|D^2\Pi\|+\Ra\|\nabla\theta\|_{L^2}\\
			&\leq \dfrac{\widehat{\beta}e^{\widehat{\beta}}}{2\pi -\widehat{\beta}}\|\nabla\theta\|_{L^2}+\dfrac{e^{\widehat{\beta}}}{4}\|\nabla\Pi\|_{L^2}+\Ra\|\nabla\theta\|_{L^2}\\
			&\leq \dfrac{\widehat{\beta}e^{\widehat{\beta}}}{2\pi -\widehat{\beta}}\|\nabla\theta\|_{L^2}+\dfrac{\pi e^{2\widehat{\beta}}}{2(2\pi-\widehat{\beta})}\|\nabla\theta\|_{L^2}+\Ra\|\nabla\theta\|_{L^2}\\
			&=(c(\widehat{\beta})+\Ra)\|\nabla\theta\|_{L^2}
		\end{split}
	\end{equation}
\end{proof}

\begin{prop}
	For all $ \varphi\in \widetilde{W}^{1,2}(\Omega_0) $, we have that
	\begin{equation}\label{key}
		\lim_N\int_0^T\left(\textbf{u}^N\cdot \nabla\theta^N-\textbf{u}\cdot\nabla\theta,\varphi\right)dt=0
	\end{equation}
\end{prop}
\begin{proof}
	Let us notice that from the Ladyzhenskaya inequality and Lemmas \ref{lemma3}, \ref{lemma6}, \ref{lemma8} it follows that
	\begin{equation}\label{gepps}
		\begin{split}
			\|(\textbf{u}^N-\textbf{u})\varphi\|_{L^2}\|\nabla\theta^N\|_{L^2}&\leq \|\textbf{u}^N-\textbf{u}\|_{L^4}\|\varphi\|_{L^4}\|\nabla\theta^N\|_{L^2}\\
			&\leq \|\textbf{u}^N-\textbf{u}\|_{L^2}^{1/2}\|\textbf{u}^N-\textbf{u}\|_{L^2}^{1/2}\|\varphi\|_{L^4}\|\nabla\theta^N\|_{L^2},
		\end{split}
	\end{equation}
	in particular from the the Rellich-Kondrachov theorem we have that $ \textbf{u}^N\rightarrow \textbf{u} $ in $ L^2 $, and hence from \eqref{gepps} we have that
	\begin{equation}\label{cippp}
		\lim_N \|(\textbf{u}^N-\textbf{u})\varphi\|_{L^2} =0.
	\end{equation}
	Moreover it is easy to see that
	\begin{equation}\label{cioppp}
		\lim_N|(\nabla(\theta^N-\theta),\varphi\textbf{u})|=0,
	\end{equation}
	since $ \theta^N\rightharpoonup\theta $ in $ L^2(0,T;\widetilde{W}^{1,2}(\Omega_0)) $. 
	Now since
	\begin{equation}\label{cippecciop}
		\begin{split}
			&|((\textbf{u}^N-\textbf{u})\cdot \nabla\theta^N,\varphi)|+|(\textbf{u}\cdot \nabla(\theta^N-\theta),\varphi)|\\
			&\qquad\qquad\qquad\qquad\leq \|(\textbf{u}^N-\textbf{u})\varphi\|_{L^2}\|\nabla\theta^N\|_{L^2}+|(\nabla(\theta^N-\theta),\varphi\textbf{u})|,		
		\end{split}
	\end{equation}
	we deduce from \eqref{cippp}, \eqref{cioppp} and \eqref{cippecciop} that
	\begin{equation}\label{key}
		\begin{split}
			&\lim_N\int_0^T\left(\textbf{u}^N\cdot \nabla\theta^N-\textbf{u}\cdot\nabla\theta,\varphi\right)dt\\
			&\qquad\qquad\qquad=\lim_N\int_0^T\left((\textbf{u}^N-\textbf{u})\cdot \nabla\theta^N+\textbf{u}\cdot \nabla (\theta^N-\theta),\varphi\right)dt=0.	
		\end{split}
	\end{equation}
\end{proof}
The following final Theorem eventually prove the existence and uniqueness of the nonlinear perturbation system \eqref{system2ABC}. 
\begin{theorem}
	$ (\theta, \textbf{u},\Pi) $ is a weak solution of the full nonlinear system \eqref{system2ABC}.
\end{theorem}
\begin{proof}
	Let un consider \eqref{system2ABC} in stream function $ \Phi $ formulation
	\begin{equation}\label{system2ABCSTREAM}
		\begin{cases}
			\Delta \Pi-\widehat{\beta}\Pi_z=\mathsf{Ra} e^{\widehat{\beta}z} \theta_z,\\
			\Delta \Phi=-\widehat{\beta}e^{-\widehat{\beta}z}\Pi_x+\mathsf{Ra} \theta_x,\\
			\theta_t+\textbf{u}\cdot \nabla\theta=\Phi_x+\Delta \theta.
		\end{cases}
	\end{equation}
	together with boundary conditions
	\begin{equation}\label{key}
		\theta=\Pi_z=\Delta \Phi=0,\quad z=0,1,
	\end{equation}
	where $ \textbf{u}=(-\Phi_z,\Phi_x) $, and let us set
	\begin{equation}\label{labella}
		\begin{split}
			\theta^N(x,z,t)&=\sum_{i=\pm 1 \atop m,n=0}^{N}A^i_{m,n}(t)\xi^i_{mn}(x,z),\\
			\Pi^N(x,z,t)&=\sum_{i=\pm 1 \atop m,n=0}^{N}B^i_{m,n}(t)\phi^i_{mn}(x,z),\\
			\Phi^N(x,z,t)&=\sum_{i=\pm 1 \atop m,n=0}^{N}C^i_{m,n}(t)\xi^i_{mn}(x,z).
		\end{split}
	\end{equation}
	Following the same procedure as in \cite{de2019existence}, substituting \eqref{labella} in \eqref{system2ABCSTREAM}, one find a non-linear system of ODEs whose solution is unique from the Peano's theorem. Once the existence and uniqueness of the truncated solution \eqref{labella} is proved, then the passage to the limit as $ N\rightarrow +\infty $ is guaranteed by the existence of converging subsequences provided in the previous preliminary results. 
\end{proof}

\bibliographystyle{unsrt}
\bibliography{main}

\end{document}